\documentclass[a4paper]{article} 

\usepackage{a4wide}

\usepackage[usenames,dvipsnames]{xcolor}
\usepackage[english]{babel} 
\usepackage[color=cyan]{todonotes}
\usepackage{float}
\usepackage{breakcites}

\usepackage{tcolorbox}
\tcbuselibrary{skins}
\usepackage{colortbl}
\usepackage{bbm}

\usepackage{mathtools}
\usepackage{tabularx}
\usepackage{array}
\usepackage{colortbl}
\usepackage{tikz,pgfplots}
\usepackage{graphicx}
\usepackage{subfigure}
\pgfplotsset{width=\linewidth,compat=newest}
\usetikzlibrary{pgfplots.groupplots}
\usepackage[framemethod=tikz]{mdframed}

\usepackage{nicefrac}

\usepackage[utf8]{inputenc}
\usepackage[T1]{fontenc}  

\definecolor{color1}{RGB}{0,139,0} 
\definecolor{color2}{RGB}{154,255,154} 

\numberwithin{equation}{section}

\newcommand{\N}{\mathbb{N}}

\usepackage{xspace}

\newcommand{\push}{\emph{push}\xspace}
\newcommand{\pull}{\emph{pull}\xspace}

\newcommand{\Push}{\emph{push}\xspace}
\newcommand{\Pull}{\emph{pull}\xspace}

\usepackage{hyperref} 
\usepackage{url}
\usepackage[framemethod=tikz]{mdframed}
\usepackage{cuted}
\usepackage{datetime}
\usepackage{bm}

\hypersetup{hidelinks,colorlinks,breaklinks=true,urlcolor=color2,citecolor=color1,linkcolor=color1,bookmarksopen=false,pdftitle={Title},pdfauthor={Author}}
\usepackage{amssymb,amsmath,amsthm,amsfonts}
\makeatletter
\def\@endtheorem{\endtrivlist}
\makeatother

\newtheoremstyle{abcd}
  {}
  {}
  {\itshape}
  {}
  {\bfseries}
  {.}
  {.5em}
  {}

\theoremstyle{abcd}
\newtheorem{theorem}{Theorem}
\numberwithin{theorem}{section}

\newtheorem{corollary}[theorem]{Corollary}

\newtheorem{lemma}[theorem]{Lemma}

\renewcommand{\d}{\textnormal{d} }
\begin{document}

\title{Asymptotics for Pull on the Complete Graph}

\author{Konstantinos Panagiotou \and Simon Reisser}


\date{Ludwig-Maximilians-Universität München\\
Monday 18$^{th}$ Oktober, 2021}


\maketitle 

\begin{abstract}
We study the randomized rumor spreading algorithm \emph{pull} on complete graphs with $n$ vertices. 
Starting with one informed vertex and proceeding in rounds, each vertex yet uninformed connects to a neighbor chosen uniformly at random and receives the information, if the vertex it connected to is informed.
The goal is to study the number of rounds needed to spread the information to everybody, also known as the \emph{runtime}.

In our main result we provide a description, as $n$ gets large, for the distribution of the runtime that involves a martingale limit.
This allows us to establish that in general there is no limiting distribution and that convergence occurs only on suitably chosen subsequences $(n_i)_{i \in \mathbb{N}}$ of $\mathbb{N}$, namely when the fractional part of $(\log_2 n_i +\log_2\ln n_i)_{i \in \mathbb{N}}$ converges.
\end{abstract}



\thispagestyle{empty} 


\section{Introduction}
Randomized rumour spreading has applications in replicated databases \cite{demers_epidemic_1987}, mobile networks \cite{Iwanicki2009}, epidemic modelling \cite{Berger2005} and even crypto-currency \cite{Patsonakis2019}. Given a graph, the algorithm/protocol \Pull works as follows.
We start by selecting a vertex and equipping it with some piece of information.
Then we proceed in rounds, in which every vertex yet uninformed connects to a neighbor chosen uniformly at random and receives the information, if the vertex it connected to is informed.
We will study the number of rounds needed to spread the information to all vertices, also known as \emph{runtime}.
For a graph $G=(V,E)$ and any vertex~$v\in V$ we denote the (random) runtime of \Pull on $G$ with starting vertex~$v$ by~$X(G,v)$.
The most basic case, and the one studied here, is to set $G=K_n$, the complete graph on~$n$ vertices.
In that case specifying the initial vertex is not necessary; we therefore just write~$X_n$ for the runtime of \Pull on $K_n$.

\paragraph{Related Work}

Randomized rumour spreading has been researched intensively since its introduction and popularization in \cite{Frieze1985,Karp2000}.
One direction of research describes the runtime of randomized rumour spreading protocols using only general graph parameters like conductance \cite{Chierichetti2018}, diameter \cite{Feige1990} or expansion \cite{Giakkoupis2014}.
A different direction is to obtain ever more precise bounds on the runtime on specific graph classes \cite{Fountoulakis2010, Clementi2016, Daknama2020}.
One major step in that direction was achieved in \cite{Doerr2017}, where the authors studied the runtime of \Pull on the complete graph. They showed that 
\begin{align*}
    \mathbb E [X_n]= \log_2 n + \log_2\ln n +O(1),
\end{align*}
as well as the related large deviation bound
\begin{align}\label{eq:LargeDev}
    P\big(|X_n- \mathbb E[X_n]|\ge r\big)\le Ae^{-\alpha r}\quad \text{for suitable } A,\alpha>0\text{ and all } r\in \mathbb N.
\end{align}
Actually, in~\cite{Doerr2017}  qualitatively similar results for several other rumor spreading protocols were shown.
In particular they studied the protocol \Push, which differs from \pull in the way the information is spread from vertex to vertex: in \push, each \emph{informed} vertex chooses a uniformly random vertex and passes the information forward if the targeted vertex is uninformed.

Regarding \push we have by now a much more precise picture.
In \cite{Daknama2021} the distribution of the runtime was described for large $n$. Let $X_n^\push$ be the runtime of $\push $ on the complete graph, $\gamma$ the Euler-Mascheroni constant, $G$ a Gumble distributed random variable with parameter~$\gamma$ and~$c$ a specific 1-periodic function with amplitude about $10^{-9}$.
Then in \cite{Daknama2021} it was shown that, as $n\to\infty$,
\begin{align*}
    \sup_{k\in \mathbb N}\left|P(X_n^\push\ge k)-P\big(\lceil\log_2 n +\ln n +G+\gamma + c(\log_2 n -\lfloor\log_2 n\rfloor)\rceil\ge k \big )\right|=o(1).
\end{align*}
This result implies, see \cite{Daknama2021}, that there is a limiting distribution only on suitable subsequences, that is on sequences $(n_i)_{i\in \mathbb N}$ such that for some $x,y\in [0,1)$ $\log_2n_i-\lfloor\log_2n_i\rfloor\to x$ and $\ln n_i-\lfloor\ln n_i\rfloor\to y$.
Moreover, it shows that the expected runtime of \push converges only on such sequences as well, with the consequence that
\begin{align*}
    \log_2 +\ln n+1.18242\le \mathbb E[X_n^\push]\le \log_2+\ln n +1.18263, \quad n\in\mathbb{N},
\end{align*}
where both bounds are (essentially) achieved for appropriate subsequences of natural numbers. This improved upon a longish list of previous papers \cite{Frieze1985, pittel_on_1987, Doerr_2014,Doerr2017}, where sharper and sharper results for the runtime of \push were derived.

All randomized rumor spreading protocols described here proceed in rounds. In~\cite{Janson1999} minimal path lengths on graphs with random edge weights were studied. If the weights are exponentially distributed, this problem is equivalent to the so-called \emph{asynchronous} \pull, where instead of having rounds each uninformed vertex chooses independently neighbours according to a rate-1 Poisson process. In~\cite{Janson1999} it was shown that, after appropriate normalization, the limiting distribution of the runtime of asynchronous \Pull on complete graphs converges and the limit is the sum of two independent Gumbel distributed random variables. 
Some more recent results on asynchronous rumor spreading are \cite{Giakkoupis2016, Panagiotou2017, Pourmiri2020}.


\paragraph{Results}
Our main result describes the distribution of the runtime of \Pull on complete graphs.
\begin{theorem}\label{MainLemma}
There is a continuous random variable $X$ such that, as $n\to\infty$,
\begin{align*}
\sup_{k\in \mathbb{N}} 
    \big|P(X_n \ge  k)-P(\lceil\log_2 n +\log_2\ln n +X \rceil\ge k)\big| = o(1).
\end{align*}
\end{theorem}
Actually, we can provide some information about $X$. To this end, let us write $I_t$ for the set of informed vertices at the start of round $t$. In particular, $|I_0| = 1$. By definition of the protocol, every uninformed vertex becomes informed in round $t$ independently with probability $|I_t|/n$. That is, in distribution
\begin{align*}
	|I_{t+1}|= |I_t|+\text{Bin}\big(n-|I_t|,|I_t|/n\big).
\end{align*}
If $|I_t| = o(n)$, then the binomial distribution is very close to being Poisson, and we thus may approximate $|I_t|$ by the sequence of random variables given by
\begin{align*}
    J_0 = 1, ~~\text{and}~~ 
	J_{t+1}= J_t+\text{Po}(J_t),~~ t \in \mathbb N_0.
\end{align*}
This sequence doubles every round in expectation, $\mathbb{E}[J_t] = 2^t$. Moreover, it is fairly easy to establish that $(H_t)_{t \in \mathbb N_0}$ with $H_t = 2^{-t}J_t$ is a martingale and uniformly integrable. Thus, the Martingale Convergence Theorem guarantees the existence of a random variable~$H$ such that~$H_t$ converges almost surely to $H$. 
In essence, $2^{-t}|I_t|$ is 'close' to $H$ for large $t$ (and $n$); we formalize this statement in Lemma~\ref{martConver} below. The random variable $X$ is then given as $X=-\log_2 H$.

From this construction of $H$ we may obtain further information about its distribution. For example, the characteristic function $\varphi$ of $H$ has the property $\varphi = \lim_{t\to\infty}\varphi_t$, where $\varphi_t$ is the characteristic function of  $H_t$, by Levy's Continuity Theorem.
Using this we establish in Section~\ref{Lemma1} that~$H$ is continuous and almost surely positive, two main ingredients in the proof of Theorem~\ref{MainLemma}. However, finding more properties of $H$, like a handy expression for its density or expressions for its moments, turned out to be a tough challenge that we leave as an open problem.

Let us denote with $(X+x)\!\!\mid_{\mathbb{Z}}$ the distribution of $X$ translated by $x$ and restricted to integers only, that is, $(X+x)\!\!\mid_{\mathbb{Z}}$ is the random variable with domain $\mathbb{Z}$ and distribution 
\begin{align*}
    P\big((X+x)\!\!\mid_{\mathbb{Z}}\ \le k\big)
    := P(X\le k-x), \quad \ k\in \mathbb{Z} .
\end{align*}
With this definition at hand and by choosing a suitable subsequence we can derive a limiting distribution from Theorem \ref{MainLemma}. 
\begin{corollary}\label{MainCorollary}
Let $x\in [0,1)$ and let $n_i$ be a strictly increasing sequence of  natural numbers such that $\log_2n_i+\log_2\ln n_i-\lfloor \log_2n_i+\log_2\ln n_i \rfloor \to x$. Then, as $i\to \infty$, in distribution
\begin{align*}
X_{n_i}-\lfloor \log_2n_i+\log_2\ln n_i \rfloor \to (X+x)\!\!\mid_{\mathbb{Z}}.
\end{align*}
\end{corollary}
This corollary warrants some further remarks.
First of all, it is not immediately clear that a sequence with the required properties exists, at least it was not to us.
Luckily it requires only moderate effort to find a suitable one. For example, we may choose
$$
    n_i
    =
    \big\lfloor\exp(W(2^{i+x}))\big\rfloor,
    \quad
    i \in \mathbb{N},
$$
where $W$ is the principal branch of the Lambert $W$ function (or product logarithm). A key property of $W$ is that $W(x)e^{W(x)}=x$ for all $x>-1/e$.
As $W(z)= (1+o(1))\ln z$ for large $z$, see for example~\cite{Corless1996}, the sequence $(n_i)_{i \in \mathbb{N}}$ is strictly increasing. 
Thus, as $i$ gets large,
$$
    2^{\log_2n_i+\log_2\ln n_i}
    = n_i\ln n_i
    = \big(1+o(1)\big) \cdot \exp\big(W(2^{i+x})\big)W(2^{i+x})=2^{i+x+o(1)}.
$$
Secondly, we can actually say more. Large deviation bounds for \emph{pull}, see \eqref{eq:LargeDev}, yield that $(X_n)^k$ is absolutely integrable for all $k\in \mathbb{N}$ and thus convergence in distribution also implies convergence of all moments. In particular, for sequences as in Corollary~\ref{MainCorollary} 
$$ 
    \mathbb{E}
    \left[\big(X_{n_i}-\lfloor \log_2n_i+\log_2\ln n_i \rfloor\big)^k\right] \to \mathbb{E}\left[\big((X+x)\!\!\mid_{\mathbb{Z}}\big)^k\right]\quad \ \forall\ k \in \mathbb{N}.
$$

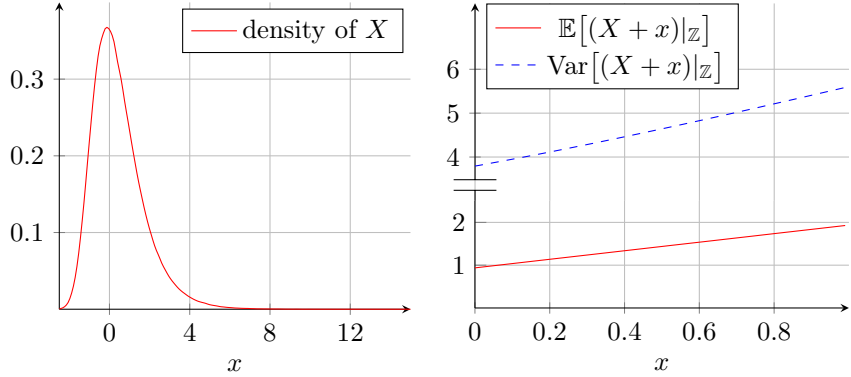
\begin{figure}[t!]
    \centering
            \begin{minipage}{0.03\textwidth}
    \end{minipage}
    \begin{minipage}{0.37\textwidth}
\centering
 		\begin{tikzpicture}
    		\begin{axis}[ymin=0,ymax=0.4, axis x line=bottom, axis y line=left, xlabel=$x$,ytick={0.1,0.2,0.3}, xtick={0,4,8,12},
    		grid=both,
    		height=5.65cm, width=6.2cm]
      			\addplot[no markers, color=red]  file {Dichte.txt};
      			\addlegendentry{{density of $X$}}
    		\end{axis}
  		\end{tikzpicture}
    \end{minipage}
    \begin{minipage}{0.09\textwidth}
    \end{minipage}
    \begin{minipage}{0.37\textwidth}
\centering
    \pgfplotsset{
}
\begin{tikzpicture}
\begin{groupplot}[
    group style={
        group name=my fancy plots,
        group size=1 by 2,
        xticklabels at=edge bottom,
        vertical sep=0pt,
    },
    width=6.5cm,
    xmin=0, xmax=1
]

\nextgroupplot[ymin=3,ymax=7.5,
               ytick={4,5,6},
               xtick={0,0.2,0.4,0.6,0.8},
               grid=both,
               x axis line style= { draw opacity=0 }, 
               axis y line=left,
               axis y discontinuity=parallel,
               height=4.2cm,
               legend style={at={(0.03,1)},anchor=north west},
               xtick style={draw=none}  
                ]
\addplot[no markers, color=red]  file {ElogH.txt};
\addlegendentry{{$\mathbb{E}\big[(X+x)\vert_{\mathbb Z} \big]$}}
\addplot[no markers, color=blue, dashed]  file {ElogH2.txt};
\addlegendentry{{$\text{Var}\big[(X+x)\vert_{\mathbb Z}\big] $}}       

\nextgroupplot[ymin=0,ymax=2.5,
               ytick={1,2},
               xtick={0,0.2,0.4,0.6,0.8},
               axis x line=bottom,
               axis y line*=left,
               xlabel=$x$,
               height=3cm,
               grid=both,
               ]
\addplot[no markers, color=red]  file {ElogH.txt};
\addplot[no markers, color=blue, dashed]  file {ElogH2.txt};                
\end{groupplot}
\end{tikzpicture}
 		    \end{minipage}
    		       \label{figure_H}
    	\begin{minipage}{0.1\textwidth}
    	\end{minipage}
		 \begin{minipage}{0.85\textwidth}
  		\caption{\small The left plot shows an estimate of the density of the random variable $X$ from Theorem \ref{MainLemma}. The right plot shows, as a function of $x \in [0,1]$, the estimated expectation and variance of the random variable $(X+x)\vert_{\mathbb Z}$ defined in Corollary \ref{MainCorollary}. }
  		\end{minipage}
\end{figure}

\noindent However, as already mentioned, extracting more information from this statement requires more detailed knowledge about the moments/the distribution of $X$ that we do not have.
On a positive side, we can provide some preliminary numerical results, see Fig.~\ref{figure_H}. To get these numbers, we have drawn $10^6$ instances of the random variable $-\log_2H_{28}$ as a substitute for the random variable $X = -\log_2H$. To approximate the density we used the \texttt{gaussian\_kde} function of Pythons Scipy package. To estimate first and second moments of $(X+x)\vert_{\mathbb{Z}}$ we used the formulas
\begin{align*}
    \mathbb{E}\big[(X+x)\vert_{\mathbb{Z}}\big]= \sum_{k\ge 1}\Big( P(X \ge k-x-1)-P(X\le -k-x)\Big)
\end{align*}
and
\begin{align*}
    \mathbb{E}\big[\big((X+x)\vert_{\mathbb{Z}}\big)^2+(X+x)\vert_{\mathbb{Z}}\big]= 2\cdot\sum_{k\ge 1} k\Big(P(X \ge k-x-1)-P(X\le -k-x)\Big) ,
\end{align*}
where we again substituted $X$ by $-\log_2H_{28}$. 

\paragraph{Outline}
The  paper is structured as follows. In the next section we give the proof of the main results, which is based on three key lemmas characterising the different phases of \emph{pull}. At first, as long as less than $n^{1/3}$ vertices are informed, \emph{pull} is best described by a branching process, which, suitably normalized, has limiting distribution $H$, see Lemma \ref{Lemma1}. After that, the protocol follows essentially a deterministic recurrence relation as described in Lemma \ref{Lemma2}. Once there are only $o\big(\sqrt{n}\big)$ uninformed vertices remaining the behaviour changes once more, in that all these vertices will be informed in one additional round, see Lemma \ref{Lemma3}.
The proof of Theorem \ref{MainLemma}, based on these lemmas as outlined, is given in Subsection \ref{proofMain}. After that, we give the short proof for Corollary~\ref{MainCorollary} in Subsection \ref{proofMainCor}. The proofs of Lemmas \ref{Lemma1}-\ref{Lemma3} can be found in Subsections \ref{SubSLemma1}-\ref{SubSLemma3}.
 
\section{Proofs}
\subsection{Proof of Theorem \ref{MainLemma}}\label{proofMain}

Our first auxiliary lemma establishes that initially -- as long as there are not too many informed vertices -- the number of informed vertices essentially doubles in each round and the deviation from perfect doubling can be described in terms of a non-trivial random variable.
\begin{lemma}\label{martConver}\label{Lemma1}
There is a continuous and almost surely positive random variable $H$ such that for all $\varepsilon>0$ there are constants $n_0,t_0\in \mathbb{N}$ such that for all $n\ge n_0$ and $t_0\le t\le \log_2(n^{1/3})$ 
$$\sup_{x\in \mathbb R}\big|P\big({2^{-t}|I_t|\ge x}\big)-P(H\ge x)\big|\le \varepsilon.$$
\end{lemma}
The proof is in Section~\ref{SubSLemma1}. From now on we fix  some $\varepsilon>0$ and set for the remainder
$$t_1 := \lfloor \log_2 (n^{1/3})\rfloor,$$
where $n \ge n_0$ is given by the previous lemma.
For rounds $t \ge t_1$ it turns out that the behaviour of \Pull can be best described by a (deterministic) recurrence relation. 
We have roughly $2^{t_1} \approx n^{1/3}$ informed vertices, enough so that it is reasonable to assume that the number of newly informed vertices in the following rounds is strongly concentrated around its expectation. 

Denote by $U_t$ the set of uniformed vertices at the start of round $t$.
Then it is easy to see that $\mathbb{E}[|U_{t+1}|\bigm\vert  I_t]=(|U_t|/n)^2n$, see also Lem.~\ref{expprec} below, and thus we expect $|U_{t_1+t}|$, given $|U_{t_1}|$, to be close to  $(|U_{t_1}|/n)^{2^t}n$.
The only thing that we have to take care of is that (small) deviations from the expectation are not  blown out of proportions when considering multiple rounds.
The next lemma, that we prove in Subsection \ref{SubSLemma2}, does exactly that.
With high probability, or abbreviated as whp, means with probability tending to 1 as $n$ tends to infinity.
\begin{lemma} \label{Lemma2}
With high probability 
$$\bigcap_{~t\ge 0} \left \{\Big ||U_{t_1+t}|-\big (|U_{t_1}|/n\big )^{2^t}n\Big |\le |U_{t_1+t}| \cdot n^{-1/50}+n^{1/4}\right \}.$$
\end{lemma}
This lemma will enable us to track the process all the way until there are fewer than $\sqrt{n}$ uniformed vertices remaining.
Indeed, as we will argue shortly, $\sqrt{n}$ is an important threshold in the following sense. On the one hand, if there are substantially more than $\sqrt{n}$ uninformed vertices, very likely the process will not terminate in the next round.
In contrast, if there are much less than $\sqrt{n}$ uninformed vertices, the process will likely terminate in the next round.
Furthermore, we will see that it is very unlikely that $|U_t|=\Theta\big(\sqrt{n}\big)$ for some $t$, where the process terminates only with constant probability.
The next lemma summarizes our findings.
\begin{lemma}\label{Lemma3}
Let $T=\min \big\{t\in \mathbb{N}: |U_t|<\sqrt n\big\}$, then with high probability 
$$|U_{T-1}|=\omega\big(\sqrt{n}\big), \quad |U_T|=o\big(\sqrt{n}\big), \quad |U_{T}|>0 \quad \text{and}\quad|U_{T+1}|=0.$$
\end{lemma}
The last statement implies immediately that whp $X_n = T + 1$ and it thus provides a handy way to compute $X_n$. To that end we utilize Lemma \ref{Lemma2} that guarantees  whp for all  $t\in \mathbb{N}$ 
\begin{align*}
|U_{t_1+t}|=\big(1+o(1)\big)\big(|U_{t_1}|/n\big)^{2^t}n+O \big(n^{1/4}\big).
\end{align*}
Together with Lemma \ref{Lemma3} this implies whp
$$\big(|U_{t_1}|/n\big)^{2^{T-t_1-1}}n=\omega\big(\sqrt{n}\big)\quad \text{and}\quad \big(|U_{t_1}|/n\big)^{2^{T-t_1}}n=o\big(\sqrt{n}\big)$$
and therefore also whp 
$$
    T
    = \min\Big\{t \in \mathbb{N}: \big(|U_{t_1}|/n\big)^{2^{t-t_1}}n < \sqrt{n}\Big\}
    = \min\Big\{t\in \mathbb{N}:(1-|I_{t_1}|/n)^{2^{t-t_1}}n< \sqrt{n}\Big\}.
$$ 
Let $T'$ be the real number such that $\big(1-|I_{t_1}|/n\big)^{2^{T'}}n= \sqrt{n}$. Then $T=\big\lfloor T'+t_1+1\big\rfloor$
and it is straightforward to verify that
\begin{align*}
    T' = \log_2n-\log_2|I_{t_1}|+\log_2 \ln n-1+o(1).
\end{align*}
Therefore, as Lemma \ref{Lemma3} yields whp $X_n=T+1= \lfloor T'+t_1+2\big\rfloor$,
\begin{align*}
\sup_{k\in \mathbb{N}}
    \left|P(X_n\ge k)-P\Big(\big\lfloor\log_2 n +\log_2\ln n -\log_2 \big(2^{-t_1}|I_{t_1}|\big) + 1+o(1)\big\rfloor\ge k\Big)\right| = o(1).
\end{align*}
In Lemma \ref{Lemma1} we showed that $2^{-t_1}|I_{t_1}|$ converges in distribution to a random variable $H$; since $H$ is continuous so is $X = -\log_2 H$ and the claim in Theorem \ref{MainLemma} follows readily. 
\subsection{Proof of Corollary \ref{MainCorollary}}\label{proofMainCor}
 Theorem \ref{MainLemma} states that there is a continuous random variable $X$ such that for all $x\in [0,1)$ and strictly increasing sequences $n_i$ such that $\log_2n_i+\log_2\ln n_i-\lfloor \log_2n_i+\log_2\ln n_i \rfloor \to x$
\begin{align*}
\sup_{k\in \mathbb{N}}
    \Big|P\big(X_{n_i}\ge k\big)-P\big(\log_2 n_i +\log_2\ln n_i +X + 1\ge k\big)\Big| = o(1).
\end{align*}
Setting $\{y\}=y-\lfloor y\rfloor$ for all $ y\in \mathbb{R}$ and substituting $k=\lfloor \log_2n_i+\log_2\ln n_i\rfloor +t+1$ we obtain
\begin{align*}
\sup_{t\in \mathbb{Z}}
    \Big|P\big(X_{n_i}\ge \lfloor \log_2n_i+\log_2\ln n_i\rfloor +t+1\big)-P\big(\{\log_2 n_i+\log_2\ln n_i\} +X \ge t \big)\Big| = o(1).
\end{align*}
Thus $X$ being a continuous random variable 
\begin{align*}
\sup_{t\in \mathbb{Z}}
    \big|P\big(X_{n_i}\ge \lfloor \log_2n_i+\log_2\ln n_i\rfloor +1+t\big)-P(x+X\ge t)\big|
= o(1).
\end{align*}
Thus
$P(X_{n_i}-\lfloor \log_2n_i+\log_2\ln n_i\rfloor\le  t)\overset{i\to\infty}{\longrightarrow} P( X\le t-x)$,
as claimed.

\subsection{Proof of Lemma \ref{Lemma1}}\label{SubSLemma1}
To prove Lemma \ref{Lemma1} first recall that
\begin{align*}
J_0=1\quad\text{ and }\quad J_{t+1}=J_t+\text{Po}(J_t),~H_t=2^{-t}J_t,\ t\in \mathbb N_0.
\end{align*}
We show three claims in order to prove Lemma \ref{Lemma1}, namely that $|I_t|$ is close to $J_t$, then that $(H_t)_{t \in \mathbb{N}_0}$ is a martingale that converges (to $H$) and finally that the limit is absolutely continuous.  

\paragraph{$|I_t|$ and $J_t$ are close.}
We begin with a simple lemma that determines the first and second moment of $J_t$.
\begin{lemma}\label{expJ_t}
For all $t\in \mathbb{N}_0$ and $J_t$ as defined above
\begin{align*}
\mathbb{E}[J_t]=2^t\quad\text{and}\quad \mathbb{E}[J_t^2]=2^{t-1}(3\cdot 2^t-1).
\end{align*}
\end{lemma}
\begin{proof}
We compute both moments inductively, starting with the base case
\begin{align*}
\mathbb{E}[J_0]=1\quad \text{and}\quad \mathbb{E}[J_0^2]=1.
\end{align*}
Moreover, using the tower property of the expectation and that $\mathbb{E}[\text{Po}(\lambda)]=\lambda$ for any $\lambda>0$ we obtain by induction
\begin{align*}
\mathbb{E}[J_{t+1}]=\mathbb{E}[J_t+\text{Po}(J_t)]=\mathbb{E}\Big[\mathbb{E}\big[J_t+\text{Po}(J_t)\bigm\vert  J_t\big]\Big]=2\cdot \mathbb{E}[J_t]=2^{t+1}.
\end{align*}
We compute the second moment similarly. Since $\mathbb{E}[\text{Po}(\lambda)^2]=\lambda+\lambda^2$ we obtain that
\begin{align*}
\mathbb{E}[J_{t+1}^2]&=\mathbb{E}\left[\big(J_t+\text{Po}(J_t)\big)^2\right] = \mathbb{E}\left[J_t^2+2\cdot J_t\text{Po}(J_t)+\text{Po}(J_t)^2\right]= 4\cdot \mathbb{E}[J_t^2]+ \mathbb{E}[J_t]\\
&=4\cdot 2^{t-1}(3\cdot 2^t-1) +2^{t}= 2^{t}(3\cdot 2^{t+1}-1).
\end{align*}
\end{proof}
The next (well-known) statement bounds the distance between two Poisson distributed random variables and furthermore quantifies the distance in the Poisson limit theorem.  
Recall that the \emph{total variation distance} for two integer valued random variables~$X,Y$ can be defined as
\begin{align}\label{eq:totalVarDist}
\d(X,Y):=\frac 12\sum_{k\in \mathbb Z}\big|P(X= k)-P(Y= k)\big|.
\end{align}
\begin{lemma}[\cite{Serfling1978}, Eq.~3.6 and Thm.~4.1]
\label{LeCam}\
\begin{itemize}
\item[a)]Let $\lambda_1,\lambda_2\in\mathbb{N}$ and $X\sim \textnormal{Po}(\lambda_1)$ and $Y\sim \textnormal{Po}(\lambda_2)$ be independent Poisson-distributed random variables. Then
$\d(X,Y)\le |\lambda_1-\lambda_2|$.
\item[b)] Let $X\sim \textnormal{Bin}(n,p)$ and $Y\sim \textnormal{Po}(np).$ Then $\d(X,Y) \le n p^2$.
\end{itemize}
\end{lemma}
With these ingredients at hand we can give a bound on the distance of $|I_t|$ and $J_t$ that is quite strong as long as $t$ is not too large.
\begin{lemma}\label{poisApp}
For all $t\in \mathbb{N}$ 
\begin{align*}
    \d_t := \d\big(|I_t|,J_t\big)\le 2\cdot 4^{t}/n.
\end{align*}
\end{lemma}\begin{proof}
Note that it suffices to consider only the case $t\le \log_4 n$, as otherwise the claimed bound is greater than one and consequently trivially true.
There are $|I_t|$ informed vertices in round $t$. Then the probability of any vertex $v\in U_t$ to be informed in that round is $|I_t|/n$ and furthermore it is independent of all other uninformed vertices, that is, the number of newly informed vertices is binomially distributed with $|U_t|$ tries and success probability $|I_t|/n$.
Thus, in distribution, 
\begin{align}\label{eq:pullBinom}
|I_{t+1}|=|I_t|+\text{Bin}\big(|U_t|,|I_t|/n\big) ,
\end{align}
an equation that we have already encountered in the introduction. 
We prove the statement of the lemma by induction over $t$. The base case is obvious as $|I_0|=1=J_0$. For the induction step, we use \eqref{eq:totalVarDist} together with $|I_t|$ and $J_t$ only taking values on the positive integers, to get 
\begin{align*}
    & 2\cdot \d_{t+1}
        = \sum_{k\ge 1}\big|P\big(|I_{t+1}|= k\big)-P(J_{t+1}= k)\big|\\
    = &\sum_{k=1}^n
        \left |
            \sum_{\ell = 1}^k P\big(|I_{t+1}|= k\bigm\vert     |I_t|=\ell\big)P\big(|I_t|=\ell\big)-P\big(J_{t+1}= k\bigm\vert J_t=\ell\big)P(J_t=\ell)
        \right |+P(J_{t+1}>n)
        .
\end{align*}
To simplify this expression we consider the auxiliary calculation
\begin{align*}
    & ~ \sum_{k= 1}^n\Big|\sum_{\ell = 1}^k P\big(|I_{t+1}|= k\bigm\vert  |I_t|=\ell\big)\big (P\big(|I_t|=\ell\big)-P(J_t=\ell)\big)\Big|\\
    \le & ~ \sum_{\ell=1}^n
    \Big|
        P\big(|I_t|=\ell\big)-P(J_t=\ell)
    \Big| \sum_{k = \ell}^nP\big(|I_{t+1}|= k\bigm\vert  |I_t|=\ell\big)\\
   \le & ~ \sum_{\ell= 1}^n\Big| P\big(|I_t|= \ell\big)-P(J_t= \ell)\Big|\le 2\cdot \d_t.
\end{align*}
In order to obtain a bound for the tail probability of $J_{t+1}$ we use Lemma \ref{expJ_t} as well as the assumption $t\le \log_4 n$ so that by Chebyshev's inequality and plenty of room to spare 
\begin{align*}
    P(J_{t+1}>n)\le P\big(|J_{t+1}-\mathbb E[J_{t+1}]|>n-\mathbb E[J_{t+1}]\big)\le\frac{\text{Var}[J_{t+1}]}{(n-\mathbb E[J_{t+1}])^2}\le4^t/n.
\end{align*}
Applying these bounds to $\d_{t+1}$ we get 
\begin{align*}
2 \cdot  \d_{t+1}\le \sum_{k=1}^n\sum_{\ell = 1}^k P(J_t=\ell)\Big|P\big(|I_{t+1}|= k\bigm\vert  |I_t|=\ell\big)-P\big(J_{t+1}= k\bigm\vert  J_t=\ell\big)\Big |+2\cdot \d_t+ 4^t/n.
\end{align*}
Next we plug in the distributions for $|I_{t+1}|-|I_{t}|$ (binomial) and $J_{t+1}$ (Poisson) to get
\begin{align*}
&2\cdot \d_{t+1} 
\le \sum_{k= 1}^n\sum_{\ell = 1}^k P(J_t=\ell)\Big |P\big(\text{Bin}(n-\ell,\ell/n)= k-\ell\big)-P\big(\text{Po}(\ell)= k-\ell\big)\Big |+2\cdot\d_t+ 4^t/n
\end{align*}
and shifting indices yields
\begin{align*}
2\cdot \d_{t+1}
\le&\sum_{\ell= 1}^nP(J_t=\ell)\sum_{k= 1}^n\Big |\big(P(\text{Bin}(n-\ell,\ell/n)= k\big)-P\big(\text{Po}(\ell)= k\big)\Big |+2\cdot \d_t+ 4^t/n\\
\le&\sum_{\ell= 1}^nP(J_t=\ell)\cdot 2\cdot \d\big(\text{Bin}(n-\ell,\ell/n),\text{Po}(\ell) \big)+2\cdot \d_t+ 4^t/n.
\end{align*}
As $\text{d}$ is a metric we can use the triangle inequality and with Lemma $\ref{LeCam}$ we get for all $0\le \ell \le n$
\begin{align*}
    \d\big(\text{Bin}(n-\ell,\ell/n),\text{Po}(\ell) \big)
    & \le \d\big(\text{Bin}(n-\ell,\ell/n),\text{Po}((n-\ell)\ell/n)\big )+\d\big(\text{Po}((n-\ell)\ell/n) ,\text{Po}(\ell) \big)\\
    &\le (n-\ell)(\ell/n)^2+\big|(n-\ell)\ell/n-\ell\big|,
\end{align*}
which is at most ${2\ell^2}/{n}$.
By plugging this into the previous inequality we  get
\begin{align*}
\d_{t+1}
\le~\sum_{\ell= 1}^nP(J_t=\ell)\frac{2\ell^2}{n}+ \d_t+ 4^t/n
\le  2\cdot \frac{\mathbb{E}[J_t^2]}n+ \d_t+ 4^t/n.
\end{align*}
Lemma~\ref{expJ_t} determines the second moment of $J_t$. By using  the induction hypothesis we conclude
\begin{align*}
\d_{t+1}
\le 3\cdot 4^{t}/n+ 2\cdot 4^{t}/n+4^{t}/n\le  2\cdot 4^{t+1}/n.
\end{align*}
\end{proof}

\paragraph{$(H_t)_{t\in \mathbb{N}_0}$ is a martingale that converges to $H$.}
Next we show that the sequence $(H_t)_{t\in \mathbb N_0}$ is a martingale and converges almost surely and in $\mathcal{L}^2$ to the random variable $H$.
\begin{lemma}\label{HConv}
There is a random variable $H$ such that $H_t\to H$ almost surely and in $\mathcal{L}^2$. 
Furthermore $H$ has mean 1 and variance $1/2.$
\end{lemma}
\begin{proof}
First we show that $H_t$ is a martingale. Let $\mathcal{F}_t$ be the filtration induced by the random variables $J_t$, then
\begin{align*}
\mathbb{E}\big[H_{t+1}\bigm\vert  \mathcal{F}_t\big]=\mathbb{E}\big[2^{-t-1}\big(J_{t}+\text{Po}(J_t)\big)\bigm\vert  \mathcal{F}_t\big]=2^{-t}J_t=H_t.
\end{align*}
Next, we show that $H_t\in \mathcal{L}^2$. Therefore we compute
\begin{align*}
\mathbb{E}\big[|H_{t}|^2\big] =2^{-2t}\mathbb{E}\big[J_t^2\big]= 2^{-2t}\cdot 2^t(3\cdot 2^{t+1}-1)\le 6
\end{align*}
using Lemma \ref{expJ_t} and consequently 
\begin{align}\label{eq:SeconMH}
\sup_{t\in\mathbb{N}_0}\mathbb{E}\big[|H_{t}|^2\big]<\infty.
\end{align}
This yields the integrability of $H_t$ and as $H_t$ is obviously measurable with respect to $\mathcal{F}_t$ we conclude that it is indeed a martingale. Thus  \eqref{eq:SeconMH} and $\mathcal{L}^p$ convergence of martingales implies the first claim, see e.g.~\cite[Thm.~4.4.6]{Durrett2019}. The values for the expectation and the variance  follow immediately from Lemma \ref{expJ_t} by scaling with $2^{-t}$ and $2^{-2t}$ respectively and then taking the limit.
\end{proof}
An even stronger version of this lemma could be shown, i.e., the martingale converges in $\mathcal{L}^p$ for all $p>1$. In any case,  the version stated here suffices for our purposes. 

In order to show the  properties  of $H$ claimed in Lemma \ref{Lemma1}  we need to describe  its characteristic function. The next lemma does exactly that, but we need a definition first. Let
\begin{align}\label{eq:h}
h(x)=h^{(1)}(x)= x e^{x-1}, \quad h^{(t+1)}=h^{(t)}\circ h,\ t\in \mathbb{N}.
\end{align}
This function is not new in the context of rumor spreading, it plays an important role in the closely related context of \cite{Daknama2021}, where it describes the evolution of the number of uninformed vertices of \Push on complete graphs. 
\begin{lemma}\label{phiItLemma}
The characteristic functions $\varphi_t$ of $H_t$ and $\varphi$ of $H$ satisfy
\begin{align*}
    \varphi_t(x)=h^{(t)}\Big(e^{ix2^{-t}}\Big)\quad \text{and}\quad \varphi(x)=\lim_{t\to \infty}\varphi_t(x).
\end{align*}

\end{lemma}
\begin{proof}
To prove the claim we first compute the probability generating function $\tilde{J}_t(x)$ of $J_t$.
Note that $\tilde{J_0}(x)=x$, as $P(J_0=1)=1$. For $t\in\mathbb{N}_0$ and $|x|\le 1$ we get 
\begin{align*}
\tilde{J}_{t+1}(x) =\sum_{k\ge 0}P(J_{t+1}=k)x^k & = \sum_{k\ge 0}P\big(J_t+\text{Po}(J_t)=k\big)x^k\\
&=\sum_{k\ge 0}\sum_{0 \le \ell \le k} P\big(\text{Po}(\ell)=k-\ell\big)P(J_t=\ell)x^k\\
&=\sum_{\ell\ge 0}P(J_t=\ell)\sum_{k\ge0}\frac{\ell ^{k}}{k!}e^{-\ell}x^{k+\ell}\\
& =\sum_{\ell\ge 0}P(J_t=\ell)e^{-\ell + x\ell}x^{\ell} \\
& =\sum_{\ell\ge 0}P(J_t=\ell)\big (x e^{x-1}\big)^\ell
=\tilde{J_t}\big(x e^{x-1}\big)=\big(\tilde{J_t}\circ h\big)(x).
\end{align*}
Thus $J_t$ has characteristic function $x\mapsto h^{(t)}\big(e^{ix}\big)$ and as~$H_t=2^{-t}J_t$ we immediately obtain that~$H_t$ has characteristic function $x \mapsto h^{(t)}\big(e^{ix2^{-t}}\big)$.
With Levy's continuity theorem we infer that the characteristic function of $H_t$ converges to the characteristic function of $H$, as Lemma~\ref{HConv} guarantees that $H_t$ converges to $H$ almost surely and therefore also in distribution. 
\end{proof}

\paragraph{Properties of $H$.} In the last part of this section we show that $H$ is absolutely continuous and almost surely positive. To show absolute continuity, we will argue that the characteristic function of $H$ is integrable. To achieve this we first find a recurrence relation for real and imaginary parts of $\varphi_t$, the characteristic function of $H_t$. We then use this description to find a second order approximation of $\varphi_t$ that eventually allows us to uniformly bound the absolute value of $\varphi_t$ by an integrable function.

The mapping that we will use to describe the real and imaginary parts of $\varphi_t$ is given by
\begin{align*}
F=F^{(1)}:\mathbb{R}^2\to \mathbb{R}^2, F\left (\binom{R}{I}\right )=e^{-1+R}\begin{pmatrix}\cos I&-\sin I\\ \sin I&~\cos I\end{pmatrix}\binom{R}{I},
\enspace \text{and} \enspace
F^{(t+1)} = F^{(t)}\circ F, \enspace t \in \mathbb{N}.
\end{align*}
Moreover, we set $F^{(0)}$ to be the identity on $\mathbb R^2$.
\begin{lemma}\label{F}Let $\varphi_t$ be the characteristic function of $H_t$. Set $I_t(x)=\textnormal{Im}\big(\varphi_t(x)\big)$ (the imaginary part), $R_t(x)=\textnormal{Re}\big(\varphi_t(x)\big)$ (the real part) and $a_t(x)=|\varphi_t(x)|$. Then for all $t\in \mathbb N_0$
\begin{align*}
\binom{R_{t}(x)}{I_{t}(x)}=F^{(t)}\left (\binom{\cos(x2^{-t})}{\sin(x2^{-t})}\right ).
\end{align*}
and
\begin{align*}
a_{t+1}(x)=a_t(x/2)\exp\big(-1+R_t(x/2)\big).
\end{align*}
\end{lemma}
\begin{proof} 
Using Lemma \ref{phiItLemma} we obtain for $t\in \mathbb{N}_0$
\begin{align}\label{eq:phiH}
\varphi_{t+1}(x)=h^{(t+1)}\left(e^{ix2^{-t-1}}\right)=h\left(h^{(t)}\left(e^{i(x/2)2^{-t}}\right)\right)=h\big(\varphi_t(x/2)\big).
\end{align}
We continue with a simple observation. 
For two complex numbers $z,w$ the imaginary and real parts of their product satisfy
\begin{align*}
    \text{Re}(z\cdot w)=\text{Re}(z)\text{Re}(w)-\text{Im}(z)\text{Im}(w)\quad \text{and}\quad \text{Im}(z\cdot w)=\text{Re}(z)\text{Im}(w)+\text{Im}(z)\text{Re}(w).
\end{align*}
Using this observation and \eqref{eq:phiH} we obtain
\begin{align*}
I_{t+1}(x)&=\text{Im}\big(\varphi_{t+1}(x)\big)=\text{Im}\Big(\varphi_{t}(x/2)\exp\big(-1+\varphi_t(x/2)\big)\Big)\\
&=R_t(x/2)\text{Im}\Big(\exp\big(-1+\varphi_t(x/2)\big)\Big)+I_t(x/2)\text{Re}\Big(\exp\big(-1+\varphi_t(x/2)\big)\Big)\\
&=\Big(R_t(x/2)\sin\big(I_t(x/2)\big)+I_t(x/2)\cos\big(I_t(x/2)\big) \Big)\exp\big(\!-1+R_t(x/2)\big)
\end{align*}
and similarly for the real part
\begin{align*}
R_{t+1}(x)&=\text{Re}\big(\varphi_{t+1}(x)\big)=\text{Re}\Big(\varphi_{t}(x/2)\exp\big(-1+\varphi_t(x/2)\big)\Big)\\
&=R_t(x/2)\text{Re}\Big(\exp\big(-1+\varphi_t(x/2)\big)\Big)-I_t(x/2)\text{Im}\Big(\exp\big(-1+\varphi_t(x/2)\big)\Big)\\
&=\Big(R_t(x/2)\cos\big(I_t(x/2)\big)-I_t(x/2)\sin\big(I_t(x/2)\big) \Big)\exp\big(\!-1+R_t(x/2)\big).
\end{align*}
Applying these two equations repeatedly and remembering that $\varphi_0(x)=e^{ix}$ and therefore $R_0(x)=\cos(x)$ as well as $I_0(x)=\sin (x)$ implies the first claim.
To show the second claim in the lemma (about $a_{t+1}$) we use again \eqref{eq:phiH} and  $|e^z|=e^{\text{Re}(z)}$ for all $z\in \mathbb C$ 
\begin{align*}
    a_{t+1}(x)=\big| \varphi_t(x/2)\cdot \exp\big(\varphi(x/2)-1\big)\big| = a_t(x/2)\cdot \exp\big(-1+R_t(x/2)\big).
\end{align*}
\end{proof}
With that recursive description at hand we can derive a (first) handy approximation for $\varphi_t$.
\begin{lemma}\label{Fseries}
Let $t\in \mathbb{N}_0$ and $x\in \mathbb R $. For all $0\le j\le \max \{j\in \mathbb N_0: |x2^{-t+j}|\le 1/16 \}$ 
\begin{align*}
 \left| F^{(j)}\left (\binom{\cos(x2^{-t})}{\sin(x2^{-t})}\right )-\binom{1-x^22^{-2t+j-2}(3\cdot2^j+1)}{x2^{-t+j}} \right|\le \binom{|x^3|2^{3(-t+j)}}{|x^2|2^{2(-t+j)}}.
\end{align*}
\end{lemma}
\begin{proof}
If $\big\{j\in \mathbb N_0: |x2^{-t+j}|\le 1/16 \big\}=\emptyset$ we have nothing to show, thus we assume that $j_{\max}:= \max\big\{j\in \mathbb N_0: |x2^{-t+j}|\le 1/16 \big\}\ge 0$. We will show the claim by induction over all $0\le j\le j_{\max} $. Very important ingredients in the forthcoming arguments are the following estimates for  smallish~$x$ that are rather easy to show:
\begin{align}\label{eq:cosexp}
\left |\cos(x)-\left (1-\frac{x^2}{2}\right )\right |\le \frac{x^4}{24}\quad \text{for all}\quad |x|\le 7 
\end{align}
and
\begin{align}\label{eq:sinexp}
\big|\sin(x)-x\big |\le \frac{x^3}{5}\quad \text{for all}\quad |x|\le 2 .
\end{align}
These estimates  yield (with quite some room to spare) immediately the induction start $(j=0)$, as by convention $F^{(0)}$ is the identity on $\mathbb R^2$.
We proceed with the induction step. For the following computations abbreviate
\begin{align*}
\alpha_j=2^{-2t+j-2}(3\cdot2^{j}+1),\quad \beta_j=2^{-t+j}\quad \text{and}\quad\Delta_j=\binom{\Delta_{j,1}}{\Delta_{j,2}}=\binom{|x^3|2^{-3t+3j}}{|x^2|2^{-2t+2j}}.
\end{align*}
In the remainder of this proof we use the following notation. For real numbers $a,b$ we write $a\pm b$ to denote \emph{some}  real number $c$ that satisfies $|a-c|\le b$. In particular, if we apply a function, e.g.,~$F$, to $a\pm b$ we understand that as $F$ applied to some number $c$ in the designated interval.
This notation is useful as we are only interested in upper and lower bounds on $F(a\pm b)$ that we can deduce from $a$ and $b$ only. 

Let $0\le j\le j_{\max}-1$. By applying the induction hypothesis we get
\begin{align}\label{eq:Fcomp}
F^{(j+1)}\left (\binom{\cos(x2^{-t}}{\sin(x2^{-t}}\right )=F\left (\binom{1-x^2\alpha_{j}}{x\beta_{j}}\pm \Delta_{j}\right )=:\binom{F_1}{F_2}.
\end{align}
Using the definition of $F$ we obtain for the first component 
\begin{align*}
F_1=(F_{11}-F_{12})F_{13},
\end{align*}
where we abbreviated 
\begin{align*}
    F_{11}&=\big(1-x^2\alpha_{j}\pm \Delta_{j,1}\big)\cos\big(x\beta_{j}\pm \Delta_{j,2}\big)\\
    F_{12}&=\big(x\beta_{j}\pm \Delta_{j,2}\big)\sin\Big(x\beta_{j}\pm \Delta_{j,2}\big)\\
    F_{13}&=\exp\big(-x^2\alpha_{j}\pm \Delta_{j,1}\big).
\end{align*}
To study these expressions we look at three recurring components first. Note that, as $(\Delta_{j,2})^2\le \Delta_{j,1}/16$ by our assumption on $j\le j_{\max}$,
\begin{align}\label{eq:AuxQuad}
    \big(x\beta_{j} \pm \Delta_{j,2}\big)^2=(x\beta_j)^2\pm   2|x|\beta_{j}\Delta_{j,2}\pm (\Delta_{j,2})^2= (x\beta_j)^2\pm  \frac{33}{16}\Delta_{j,1}.
\end{align}
Furthermore, using again $|x|2^{-t+j}\le 1/16$ guaranteed by $j\le j_{\max}$,
\begin{align}\label{eq:AuxCub}
    \big|\big(x\beta_{j}\pm  \Delta_{j,2}\big)^3\big|\le  \sum_{i=0}^3\binom{3}{i}(|x|\beta_{j})^i(\Delta_{j,2})^{3-i}\le\Delta_{j,1}\left(1+\frac{3}{16}+\frac{3}{16^2}+\frac{1}{16^3}\right)\le  \frac{6}5\Delta_{j,1}
\end{align}
and similarly
\begin{align}\label{eq:AuxQuart}
   \big|\big(x\beta_{j}\pm \Delta_{j,2}\big)^4\big |\le \sum_{i=0}^4\binom{4}{i}(|x|\beta_{j})^i(\Delta_{j,2})^{4-i}\le \frac{4}{50}\Delta_{j,1}.
\end{align}
Lastly, we use again $|x|2^{-t+j}\le 1/16$ and $x^4a_j^2\le \Delta_{j,1}/16$ to bound \begin{align}\label{eq:AuxQuadAj}
    \big|\big(x^2\alpha_{j}\pm\Delta_{j,1}\big)^2\big|\le |x|^4\alpha_j^2+|x|^2\alpha_j\Delta_{j,1}+\Delta_{j,1}^2\le \Delta_{j,1}\left(\frac{1}{16}+\frac{1}{16^2}+\frac{1}{16^3}\right)\le \frac{1}{8}\Delta_{j,1}.
\end{align}
Combining these bounds with \eqref{eq:cosexp} and \eqref{eq:sinexp} we will obtain estimates for the $\sin$ and $\cos$ terms in $F_{11}$ and $F_{12}$. By   \eqref{eq:cosexp} 
\begin{align*}
    \cos\big(x\beta_{j}\pm \Delta_{j,2}\big)=1-\frac{(x\beta_j\pm \Delta_{j,2})^2}2\pm \frac{(x\beta_j\pm \Delta_{j,2})^4}{24} 
  \end{align*}
and by combining this with \eqref{eq:AuxQuad} and \eqref{eq:AuxQuart} we obtain that
\begin{align}\label{eq:AuxCos}
   \cos\big(x\beta_{j}\pm \Delta_{j,2}\big)=1-\frac{(x\beta_j)^2}2\pm \frac{17}{16}\Delta_{j,1}.
\end{align}
In the same way, using \eqref{eq:sinexp} and \eqref{eq:AuxCub},
\begin{align}\label{eq:AuxSin}
    \sin\big(x\beta_{j}\pm \Delta_{j,2}\big)=(x\beta_j\pm \Delta_{j,2})\pm \frac{\big(x\beta_{j}\pm \Delta_{j,2}\big)^3}5 = x\beta_j\pm \Delta_{j,2} \pm \frac{\Delta_{j,1}}{4}.
\end{align}
Having done these preparations we proceed with deriving  bounds for $F_{11},F_{12},F_{13}$.
We begin with~$F_{11}$ and using \eqref{eq:AuxCos} as well as $|x|2^{-x+j}\le 1/16$ we obtain that
\begin{align*}
F_{11}= \left (1-x^2\alpha_{j}\pm\Delta_{j,1}\right )\left (1-\frac{(x\beta_{j})^2}{2}\pm  \frac{17}{16}\Delta_{j,1}\right )= 1-x^2\left (\alpha_j+\frac{\beta_{j}^2}{2}\right )\pm \frac{17\Delta_{j,1}}{8}.
\end{align*}
Furthermore, by making use of \eqref{eq:AuxSin} and $|x|2^{-x+j}\le 1/16$ we obtain 
\begin{align*}
F_{12} =\left (x\beta_{j}\pm\Delta_{j,2}\right )\left (x\beta_{j}\pm\Delta_{j,2} \pm \frac{\Delta_{j,1}}4 \right ) = (x\beta_j)^2\pm \frac{17\Delta_{j,1}}8
\end{align*}
and using the estimate $|\exp(x)-(1+x)|\le x^2$, valid for all $|x|\le 1$, as well as \eqref{eq:AuxQuadAj}, we get 
\begin{align*}
F_{13}= 1-x^2\alpha_{j}\pm\Delta_{j,1}\pm\big(x^2\alpha_{j}\pm\Delta_{j,1}\big)^2=1-x^2\alpha_j\pm \frac{9\Delta_{j,1}}{8}.
\end{align*}
Thus, putting $F_{11},F_{12}$ and $F_{13}$ together and using oce more  $|x|2^{-t+j}\le 1/16$, we get that
\begin{align*}
F_1&=\left (1-x^2\left (\alpha_j+\frac{3\beta_{j}^2}{2}\right )\pm \frac{34\Delta_{j,1}}{8}\right )\left (1-x^2\alpha_j\pm \frac{9\Delta_{j,1}}{8}\right )\\
&=1-x^2\left (2\alpha_j+\frac{3\beta_{j}^2}{2}\right )\pm 6\Delta_{j,1}
=1-x^2\alpha_{j+1}\pm \Delta_{j+1,1}
\end{align*}
confirming the induction step on the first component in \eqref{eq:Fcomp}. Going forward we switch our attention to the second component, which we again split into three parts
\begin{align*}
F_2=(F_{21}-F_{22})F_{13},
\end{align*}
where
\begin{align*}
    F_{21}&=\big(1-x^2\alpha_{j}\pm\Delta_{j,1}\big)\sin\big(x\beta_{j}\pm\Delta_{j,2}\big)\\
    F_{22}&=\big(x\beta_{j}\pm\Delta_{j,2}\big)\cos\big(x\beta_{j}\pm\Delta_{j,2}\big).
\end{align*}
Similarly, as above using \eqref{eq:AuxSin}, \eqref{eq:AuxCos} and $|x|2^{-t+j}\le 1/16$ we extend these expressions. We start with $F_{21}$
\begin{align*}
F_{21}&= \left (1-x^2\alpha_{j}\pm\Delta_{j,1}\right )\left (x\beta_{j}\pm\Delta_{j,2} \pm \frac{\Delta_{j,1}}4 \right )= x\beta_j\pm \frac{9\Delta_{j,2}}{8}.
\end{align*}
Continuing with $F_{22}$, again applying \eqref{eq:AuxCos} and $|x|2^{-t+j}\le 1/16$,
\begin{align*}
F_{22}=\left (x\beta_{j}\pm\Delta_{j,2}\right )\left (1-\frac{(x\beta_{j})^2}{2}\pm  \frac{17\Delta_{j,1}}{16}\right ) = x\beta_j\pm \frac{17\Delta_{j,2}}{16}.
\end{align*}
Finally, we combine $F_{21},F_{22}$ and $F_{13}$ and with $|x|2^{-t+j}\le 1/16$ we get
\begin{align*}
F_2&=\left (2x\beta_j\pm \frac{33\Delta_{j,2}}{16}\right )\left (1-x^2\alpha_j\pm \frac{9\Delta_{j,1}}{8}\right )=2x\beta_j\pm 3\Delta_{j,2}
=x\beta_{j+1}\pm \Delta_{j+1,2}.
\end{align*}
Thus we confirmed the induction step and conclude the proof.
\end{proof}
The next lemma bounds the absolute value  of $\varphi_t(x)$ for large values of $x$ and $t$ implying that $\varphi$ is integrable, a sufficient condition for the absolute continuity of $H$. We did not make any effort to optimize the involved constants.
\begin{lemma}
\label{HCharInt}
For all $x\in \mathbb{R},\ |x|\ge 2^{2^{17}}$  and $t\in \mathbb{N},\ t\ge \log_2(16|x|)$, 
\begin{align*}
|\varphi_t(x)|=a_t(x)\le |x|^{-1.2}.
\end{align*}
\end{lemma}
\begin{proof}
Let $\delta=1/32$. As $|x| \ge 1$  there is some $t_0\ge 0$ such that $x2^{-t_0}\in [\delta,2\delta]$. To be completely explicit,
\[
    t_0 := \lceil\log_2(16|x|)\rceil.
\]
Thus $t\ge t_0$  and set $j^\star= t-t_0\ge 0$. Then with Lemma  \ref{Fseries} and $\sigma(x)$ denoting the sign of $x$
\begin{align}\label{eq:Bound1}
 \binom{F_1}{F_2}:=F^{(j^\star)}\left (\binom{\cos(x2^{-t})}{\sin(x2^{-t})}\right )\le\binom{1-5\delta^2/8}{\sigma(x)\cdot2\delta} + \binom{(2\delta)^3}{(2\delta)^2}\le \binom{1-\delta^2/2}{\sigma(x)\cdot 2\delta+(2\delta)^2}.
\end{align}
Furthermore observe that by the definition of $F$ and the facts that $0\le \cos i\le 1$ and $0\le i \sin i$ for all  $i\in [-\pi/2,\pi/2]$
\begin{align}\label{eq:Bound2}
F\left (\binom{r}{i}\right )= e^{-1+r}\binom{r\cos i-i\sin i}{r\sin i+i\cos i}\le \binom{re^{-1+r}}{\pi/2}\quad \text{for all } r\in [0,1],\ i\in [-\pi/2,\pi/2].
\end{align}
Moreover, recall from Lemma \ref{F} that
$$
    \binom{R_i(x2^{-t+i})}{I_i(x2^{-t+i})} = F^{(i)}\left (\binom{\cos(x2^{-t})}{\sin(x2^{-t})}\right )
    \quad\text{for}\quad
    R_t(x) = \text{Re}(\varphi_t(x))
    \text{ and }
    I_t(x) = \text{Im}(\varphi_t(x)).
$$
Thus we can bound $R_i(x2^{-t+i}),\ i\ge j^\star$ 
by using~\eqref{eq:Bound2} for the first $i-j^\star$ applications of $F$ and~\eqref{eq:Bound1} for the remaining $j^\star$ to infer that
\begin{align*}
    \binom{R_i(x2^{-t+i})}{I_i(x2^{-t+i})} \le  \binom{F_1\cdot e^{ (i-j^\star)\cdot (-1+F_1)}}{\pi/2}\le  \binom{(1-\delta^2/2)\cdot e^{- (i-j^\star)\cdot \delta^2/2}}{\pi/2},\quad i\ge j^\star.
\end{align*}
In particular  
\begin{align}\label{eq:R_j}
R_i(x2^{-t+i})\le 1\quad \text{and} \quad R_{i^\star}\big(x2^{-t+i^\star}\big)\le e^{-2}\quad \text{for all}\quad   i\ge 0 \text{ and } i^\star\ge j^\star+4/\delta^2.
\end{align}
Now we switch our focus to $a_t(x)$. By applying Lemma~\ref{F} and~\eqref{eq:R_j}
\begin{equation*}\label{eq:a_tBound}
\begin{aligned}
a_{t}(x)&=a_{t-1}(x/2)\cdot e^{-1+R_{t-1}(x/2)}\le \exp\left ({\sum_{i=0}^{t-1}\big(-1+R_i(x2^{-t+i})\big)}\right )\\
&\le \exp\left ({\sum_{i=j^\star+4/\delta^2}^{t-1}\big(-1+R_i(x2^{-t+i})\big)}\right )
\le \exp\left ({\sum_{i=j^\star+4/\delta^2}^{t-1}\big(-1+e^{-2}\big)}\right ).
\end{aligned}
\end{equation*}
Note that the sum in the exponential is non-empty, as the definition of $t_0$ implies that $j^\star \le t-\log_2\big(|x|/(2\delta)\big )$ and as $|x|>2^{2^{17}}$  we have $\log_2\big(|x|/(2\delta)\big )>4/\delta^2+1$.
Thus
\begin{align*}
a_{t+1}(x)&\le  \exp\Big (\big(\log_2(|x|/(2\delta))-4/\delta^2\big)\big(-1+e^{-2}\big)\Big )\\
&\le  \exp \Big (-\big(4/\delta^2-\log_2(2\delta)\big)\big(-1+e^{-2}\big)\Big)\cdot |x|^{(-1+e^{-2})/\ln 2}.
\end{align*}
This implies that  $a_t(x)\le |x|^ {-1.2}$, since  numerically $(-1+e^{-2})/\ln 2 \le -1.24$ and for all $x$ with $|x|>2^{2^{17}}$ additionally $\exp (-(4/\delta^2-\log_2(2\delta))(-1+e^{-2}))\cdot  |x|^{-0.04}\le 1$.
\end{proof}
A close inspection of the previous proof suggests that actually $a_t(x) \sim |x|^{-c}$ with $c = 1/\ln 2 \approx 1.44$. This would be interesting (and it is harder) to prove and it may have important consequences, but we will not need that; for our purpose it is enough to know that $|\varphi|$ is integrable.
Next we prove the last  remaining claim in Lemma \ref{Lemma1}.
\begin{lemma}\label{HCont}\label{Hpos}\label{HBound}
$H$ is absolute continuous and  almost surely positive.
\end{lemma}
\begin{proof}
$\varphi$ is a characteristic function and therefore bounded by 1. Thus by Lemma \ref{HCharInt} it is integrable and this implies the absolute continuity of $H$.
Next we argue that $H$ is almost surely positive. We have shown that $H_t\ge 0$ for all $t$ and as $H_t\overset{t\to\infty}{\longrightarrow}H$ it follows that $H\ge 0$. Furthermore we have just shown that $H$ is indeed a continuous random variable and therefore $P(H=0)=0$ and consequently $H>0$ almost surely.

\end{proof}
\paragraph{Conclusion.}
We have shown that $2^{-t}|I_t|$ is close to $H_t$, which is a martingale that converges to $H$, a continuous random variable. Concluding this subsection we infer Lemma \ref{Lemma1} from these statements.

\begin{proof}[Proof of Lemma \ref{Lemma1}.]
In Lemma \ref{poisApp} we have shown that for all $t\in \mathbb N$
$$ \d(|I_t|,J_t) =\d(2^{-t}|I_t|,2^{-t}J_t)\le 2\cdot 4^t/n.$$ Moreover in Lemma \ref{HCont} we have shown convergence of $2^{-t}J_t$ to $H$ in $\mathcal{L}^2$ and therefore also in distribution. Thus for all $\varepsilon>0$
$$\sup_{x\in \mathbb R}\big|P\big({2^{-t}|I_t|\ge x}\big)-P\big({2^{-t}J_t\ge x}\big)\big|\le \varepsilon/2$$
as well as
$$\sup_{x\in \mathbb R}\big|P\big({2^{-t}J_t\ge x}\big)-P(H\ge x)\big|\le \varepsilon/2.$$
and therefore the claimed convergence follows by the triangle inequality.
Lemma \ref{HCont} shows the final claim: $H$ is continuous and almost surely positive.
\end{proof}
Now that we have proven Lemma \ref{Lemma1}, we state and prove a simple corollary for later reference. 
\begin{corollary}\label{ProbCond}
Let $t_1=\lfloor (1/3)\log_2 n\rfloor$. Then with high probability 
$
    |I_{t_1}| = \Theta\big(n^{1/3}\big).
$
\end{corollary}
\begin{proof}
Lemma \ref{Lemma1} yields that
\begin{align*}
    P\big(|I_{t_1}| = o(n^{1/3})\big)
    \le P\big(H =  2^{-t_1}o(n^{1/3})\big) + o(1)
    = P\big(H = o(1)\big) +o(1).
\end{align*}
Since $H$ has a density this is $o(1)$. Moreover, by applying again Lemma~\ref{Lemma1}
\begin{align*}
    P\big(|I_{t_1}| = \omega(n^{1/3})\big)
    \le P\big(H = 2^{-t_1}\omega(n^{1/3})\big) + o(1)
    = P\big(H =  \omega(1)\big)+o(1).
\end{align*}
However, since $P(H \ge h) \to 0$ when $h\to \infty$ the proof is completed.
\end{proof}
\subsection{Proof of Lemma \ref{Lemma2}}\label{SubSLemma2}

We begin with a simple lemma that determines the expected number of informed and uninformed vertices after a given round. 
\begin{lemma}\label{expprec}
For any $t \in \mathbb{N}_0$ 
\begin{align*}
 \mathbb{E}\big[|U_{t+1}|\bigm\vert  I_t\big]=\big(|U_t|/n\big )^2 \, n\quad \text{and}\quad \mathbb{E}\big[|I_{t+1}|\bigm\vert  I_t\big]=2|I_t|-|I_t|^2/n.
\end{align*}
\end{lemma}
\begin{proof}
From the definition of \Pull we know $|I_{t+1}|= |I_t|+\text{Bin}\big(n-|I_t|,|I_t|/n\big )$, see also \eqref{eq:pullBinom}, thus
\begin{align*}
 \mathbb{E}\big[|I_{t+1}| \bigm\vert  I_t\big]= |I_t|+ \big (n-|I_t|\big)\cdot \frac{|I_t|}{n}= 2|I_t|-\frac{|I_t|^2}{n}.
\end{align*}
Using the relation $|I_t|=n-|U_t|$  yields directly the second claim.
\end{proof}
A key property that simplifies greatly the computations in this section is the following observation, in a similar form  introduced in \cite{Daknama2020} and also applied in \cite{Daknama2021}.
\begin{lemma}\label{lemma_var}
For any $t\in \mathbb{N}_0$
\begin{align*}
\textnormal{Var}\big[|I_{t+1}|\bigm\vert  I_t\big]
\le\min\big\{ \mathbb{E}\big[|I_{t+1}|\bigm\vert  I_t\big],  \mathbb{E}\big[|U_{t+1}|\bigm\vert  I_t\big]\big\}.
\end{align*}
\end{lemma}
\begin{proof}
As $|I_{t+1}|= |I_t|+\text{Bin}\big(|U_t|,|I_t|/n\big )$ and $|U_t|=n-|I_t|$, 
\begin{align*}
\textnormal{Var}\big[|I_{t+1}|\bigm\vert  I_t\big]
=
\text{Var}\big[|I_t| + \text{Bin}(|U_t|,|I_t|/n) \bigm\vert I_t\big]
=
\text{Var}\big[\text{Bin}(|U_t|,|I_t|/n) \bigm\vert I_t\big]
= \frac{|U_t|^2\cdot |I_t|}{n^2}.
\end{align*}
This is obviously bounded from above by $ \mathbb{E}\big[|U_{t+1}|\bigm\vert  I_t\big]=|U_t|^2/n$ as well as by $\mathbb{E}\big[|I_{t+1}|\bigm\vert  I_t\big]=|I_t|+|U_t|\cdot |I_t|/n$. 
\end{proof}
Lemma \ref{lemma_var} and Chebychev's inequality ensure that the number of informed vertices is highly concentrated around its expectation as soon as enough vertices are informed. Compare the next lemma to \cite[Lem.~3.4]{Daknama2021} for a similar statement for \emph{push}.
\begin{lemma} 
\label{concentration_chernoff}
\label{3}
Let $t_1=\lfloor \log_2(n^{1/3}) \rfloor$. For $t \in \N, \ 0 <\varepsilon<1/4$ let $C_t$ denote the event
$$
    \Big||I_{t+1}|-\mathbb{E}\big[|I_{t+1}|\bigm\vert  I_t\big]\Big|
    \leq
    M(I_t)^{1/2+\varepsilon}+ n^{\varepsilon},
    ~\text{where}~
    M(I_t)=\min \big\{\mathbb{E}\big[|I_{t+1}|\bigm\vert  I_t\big],\mathbb{E}\big[|U_{t+1}|\bigm\vert  I_t\big]\big\}.
$$ 
Then
\begin{align*}
P\left(\bigcap\limits_{t\ge t_1 }C_t\bigm\vert  I_{t_1}\right)= 1-o(1).
\end{align*}
\end{lemma}
\begin{proof}
Observe that Corollary \ref{ProbCond} implies whp 
\begin{align}\label{eq:Start}
   \mathbb{E}\big[|I_{t+1}|\mid I_t\big]\ge  |I_t|\ge |I_{t_1}|\ge n^{1/4}\quad \text{for all} \quad t\ge t_1.
\end{align}
Set $n' := n - n^{1/2 + \varepsilon/3}$. Then
\[
    P\left(\overline{C_t}\bigm\vert I_{t}\right)
    = P\left(\overline{C_t}\mathbf{1}_{|I_t|> n'} \bigm\vert I_{t}\right) + P\left(\overline{C_t}\mathbf{1}_{|I_t|\le n'} \bigm\vert I_{t}\right).
\]
If $|I_t|> n'$, then $|U_t| < n^{1/2 + \varepsilon/3}$ and so $M(I_t) \le \mathbb{E}[|U_{t+1}|~\vert~  I_t] = |U_{t}|^2/n \le n^{2\varepsilon/3}$. In that case $\overline{C_t}$ thus implies that $|U_{t+1}| \ge n^\varepsilon \ge n^{\varepsilon/3}\mathbb{E}[|U_{t+1}|~\vert~  I_t]$. By Markov's inequality
\[
    P\left(\overline{C_t}\mathbf{1}_{|I_t|> n'} \bigm\vert I_{t}\right)
    \le
    P\big(|U_{t+1} \ge n^{\varepsilon/3}\mathbb{E}[|U_{t+1}|~\vert~  I_t] \bigm\vert I_{t}\big)
    \le 
    n^{-\varepsilon/3}.
\]
If $|I_t|\le n'$, then $|U_t| \ge n^{1/2 + \varepsilon/3}$ and using \eqref{eq:Start} also $M(I_t) \ge  \min\{ n^{1/4},|U_{t}|^2/n\} \ge n^{2\varepsilon/3}$.
In this case, using Lemma~\ref{lemma_var}, $\overline{C_t}$  implies that
$$
    \Big||I_{t+1}| - \mathbb{E}\big[|I_{t+1}|\bigm\vert  I_t\big]\Big|> M(I_t)^{1/2 + \varepsilon} \ge n^{2\varepsilon^2/3} \text{Var}\big[|I_{t+1}|\bigm\vert  I_{t}\big]^{1/2}.
$$ By Chebychev's inequality
\begin{align*}
 P\left(\overline{C_t}\mathbf{1}_{|I_t|\le n'} \bigm\vert I_{t}\right)
&\le ~P\left (\Big||I_{t+1}| - \mathbb{E}\big[|I_{t+1}|\bigm\vert  I_t\big]\Big|> n^{2\varepsilon^2/3} \text{Var}\big[|I_{t+1}|\bigm\vert  I_{t}\big]^{1/2}\Bigm\vert I_{t}\right )\le n^{-\varepsilon^2}.
\end{align*}
By combining both cases we get the very crude bound
\begin{align}\label{eq:allTe}
 P\big(\overline{ C_t}\bigm\vert  I_t\big)\le n^{-\varepsilon^2}\quad \text{for all } t\ge t_1.
\end{align}
The large deviation bounds~\eqref{eq:LargeDev} give us that $X_n\ge 2\log_2n$ has exponentially small probability. Thus
$$P\left( \bigcup_{t\ge 2\log_2 n} \overline{C_t}\mid I_{t_1}\right)= o(1).$$
A union bound and \eqref{eq:allTe}, applied to $O(\log n)$ many $t_1 \le t\le t_2$, then yield the claim. 
\end{proof}
Lemma \ref{concentration_chernoff} shows that $|I_{t}|$ is closely concentrated around its (conditional) expectation in all rounds.
This translates directly to concentration of $|U_{t+1}|$ around $(|U_t|/n)^2n = \mathbb{E}[|U_{t+1}|\bigm\vert  U_{t}]$ for all $t\ge t_1$.
Using this, we are now ready to prove Lemma \ref{Lemma2}, that is, $|U_{{t_1}+t}|$ is close to $(|U_{t_1}|/n)^{2^t}n$ for all $t\in \mathbb{N}$ with high probability.
\begin{proof}[Proof of Lemma \ref{Lemma2}]
We assume that $|U_{t_1}|=\Theta\big(n^{1/3}\big )$, which we know from Corollary \ref{ProbCond} has high probability.
Consequently we can apply Lemma \ref{concentration_chernoff} with $\varepsilon=1/10 $ and thus we get with high probability for all $t\ge t_1$ 
\begin{align}\label{eq:concentration}
\big||I_{t+1}|-\mathbb{E}\big[|I_{t+1}|\bigm\vert  I_t\big]\big|\le \Big(\min \big\{\mathbb{E}\big[|I_{t+1}|\bigm\vert  I_t\big],\mathbb{E}\big[|U_{t+1}|\bigm\vert  I_t\big]\big\}\Big)^{3/5}+n^{1/10}.
\end{align} 
For the rest of this proof we assume in addition~\eqref{eq:concentration}, that is, we assume that $(|I_t|)_{t \ge t_1}$ (and thus also $(|U_t|)_{t \ge t_1}$ and $(\mathbb{E}[|I_{t+1}|\bigm\vert  I_t])_{t \ge t_1}$) are sequences of numbers with the aforementioned properties.
In particular, \eqref{eq:concentration} implies for all $\delta>0$ that $\big|2|I_t|-|I_{t+1}|\big|\le \delta |I_{t}|$ for all $t\ge\ t_1$ and $n>\delta^{-15}$, where $t_1 = \lfloor \log_2(n^{1/3})\rfloor$. Therefore
\begin{align}\label{eq:concAppl}
|I_{t_1+s}|\le (2+\delta)^s|I_{t_1}| \quad \text{for all}\ s\in \mathbb N_0,\ \delta>0\text{ and } n>\delta^{-15}.
\end{align}
Set
$$
    \beta_{t_1+s}
    := \big(|U_{t_1}|/n\big)^{2^{s}}, \quad s\in \mathbb{N}_0.
$$
Note that $t_1+s$ is just a different way to parameterize $t\ge t_1$, which simplifies the notation when~$t_1$ is involved. In particular $t_1$ is always fixed to the aforementioned value.

We will next argue that for all $t \ge t_1$, abbreviating $\Delta_{t}:=\big ||U_{t}|-\beta_{t}n\big|$,
\begin{equation}
\label{eq:auxassumption}
\Delta_{t+1}\le \Big(\min \big\{2|I_t|,|U_t|^2/n\big\}\Big)^{3/5}+n^{1/10}+ \Big (2|U_t|/n+\Delta_t/n\Big)\Delta_t.
\end{equation}
To see this, note first that by using \eqref{eq:concentration} and Lemma \ref{expprec},
\begin{align*}
\big ||U_{t+1}|-(|U_{t}|/n)^2n\big|= \big ||I_{t+1}|-\mathbb{E}\big[|I_{t+1}|\bigm\vert  I_t\big ]\big| \le \Big(\min \big\{2|I_t|,|U_t|^2/n\big\}\Big)^{3/5}+n^{1/10}.
\end{align*}
Secondly, applying the triangle inequality, i.e., $|x+y|\le 2|x|+|x-y|$ for all $x,y\in \mathbb{R}$, yields  
\begin{align*}
\big| (|U_{t}|/n)^2n -\beta_{t+1}n\big|& = \big ||U_t|^2/n-\beta_t^2 n\big|=\big ||U_t|/n+\beta_t \big|\cdot \big||U_t|-\beta_tn \big|\\
&\le \Big(2|U_t|/n+\big ||U_t|/n-\beta_t\big|\Big)\big ||U_t|-\beta_t n\big|.
\end{align*}
The triangle inequality then implies \eqref{eq:auxassumption}. 

In the remainder of this proof we will look at the bound of $\Delta_t$ in \eqref{eq:auxassumption} in three different ways to distinguish in each case a different behaviour. Just to wit, at first $\Delta_t$ doubles as long as the number of informed vertices doubles. However,  as soon as the doubly exponential shrinking of the uninformed vertices takes over, also the error $\Delta_t$ shrinks rapidly, so that $\Delta_t$ always remains $o\big(|U_t|\big )$.  In end we just make sure that $\Delta_t$ stays small and does not increase any more. 

We will make this outline more precise by formulating matching claims, which we then use to infer the statement of this lemma. We prove the claims afterwards. Our first claim is that for $\delta = 1/100$ and $d=400$
\begin{equation}\tag{$\text{Claim 1}$}\label{eq:recI}
\Delta_{t_1+s}\le d(2+\delta)^s|I_{t_1}|^{3/5}\quad \text{for all}\ 0\le s\le  (14/15)\log_2 n-t_1\text{ and } n>\delta^{-15}.
\end{equation}
By using that $|I_{t_1}|=\Theta (n^{1/3}$ and that $14/15-1/3+1/5=4/5<0.85$,
\eqref{eq:recI} implies for sufficiently large $n$ that for $t_2=\lfloor(14/15)\log_2n\rfloor$
\begin{align}\label{eq:Delta1}
\Delta_{t_1+s}\le d  \left (2+\delta\right )^{s+1}|I_{t_1}|^{3/5}\le n^{0.85}\quad \text{for all } 0\le s\le t_2-t_1.
\end{align}
Moreover, \eqref{eq:concAppl} yields 
\begin{align*}
|I_{t_1+s}|& \le \left (2+\delta\right )^{s+1}|I_{t_1}|= o(n) \quad\text{for all } 0\le s\le t_2-t_1,
\end{align*}
and therefore, as $|U_{t_1+s}|=n-|I_{t_1+s}|=\Theta(n)$, 
\begin{align}\label{eq:baseU}
\Delta_{t_1+s}\le |U_{t_1+s}|\cdot n^{-1/10} \quad \text{for all }0\le s\le t_2-t_1.
\end{align}
From \eqref{eq:LargeDev} we know that whp  $X_n\le \log_2 n+2\log_2\ln n $, thus we need to bound $\Delta_t$ for at most an additional $(1/15)\log_2n+2\log_2\ln n$ steps.
For these steps we will need a different bound, as the bound in \eqref{eq:recI} is only useful as long as $t_1+s < \log_2 n$; otherwise the term $(2+\delta)^s$ blows up.
Our next claim is
\begin{equation}\tag{$\text{Claim 2}$}\label{eq:recU}
\Delta_{t_2+s}\le |U_{t_2+s}|\cdot (2+\delta)^s\cdot n^{-1/11}\quad \text{for all }  s\in\mathbb N\ \text{with}\ |U_{t_2+s}|\ge n^{1/4} \text{ and } n\ge \delta^{-111}.
\end{equation}
We saw that it suffices to apply \eqref{eq:recU} for at most $s\le (1/15)\log_2n+2\log_2\ln n $ additional steps. For such $s$ we obtain that $(2+\delta)^{s+1}n^{-1/11}\le n^{-1/50}$ for sufficiently large $n$ and therefore we conclude from \eqref{eq:baseU} for all $t\le t_2$ and  \eqref{eq:recU} for all $t> t_2$  
\begin{equation*}
    \Delta_{t}
    \le |U_{t}|\cdot n^{-1/50}
    \quad \text{for all}
    \quad 
    t\ge t_1
    \quad
    \text{as long as}
    \quad
    |U_{t}|\ge n^{1/4}.
\end{equation*}
For $t\ge t_1$ such that $|U_t|\le n^{1/4}$ the claim of the Lemma follows trivially.

Now we show \eqref{eq:recI} and \eqref{eq:recU}, starting with \eqref{eq:recI}, which we are going to do by induction. 
The base case follows directly from \eqref{eq:auxassumption}. Let $t\ge t_1$, from \eqref{eq:auxassumption} and using that $|I_t|=\Omega\big(n^{1/3}\big)$ we get that
\begin{align*}
\Delta_{t+1}&\le 2|I_{t}|^{3/5}+\big(2+\Delta_{t}/n\big)\Delta_{t}.
\end{align*}
Observe that by using the induction hypothesis, we obtain that  
\begin{align*}
\Delta_{t_1+s}\le |U_{t_1+s}|\cdot n^{-1/10}
\end{align*}
and thus using \eqref{eq:concAppl}, the induction hypothesis, and $n\ge \delta^{-15}$ gives us
\begin{align*}
\Delta_{t_1+s+1}&\le 2\big((2+\delta)^s|I_{t_1}|\big)^{3/5} +(2+\delta/2)\left(d(2+\delta)^s|I_{t_1}|^{3/5}\right )\\
&\le \left (\frac{2}{d}(2+\delta)^{-2s/5}+2+\frac{\delta}{2}\right )d(2+\delta)^{s}|I_{t_1}|^{3/5}.
\end{align*}
Our choice of $d$ and $\delta$ guarantees that ${(2/d)}(2+\delta)^{(-2/5)s}\le  {(2/d)}(2+\delta)^{(-2/5)}\le \delta/2$ for all $s\in \mathbb N$ and \eqref{eq:recI} follows. 

We continue with \eqref{eq:recU} that we we will prove by induction, too. 
To that end, we observe that \eqref{eq:concentration} also implies that as long as $|U_{t+1}|\ge n^{1/4}$ and $n\ge \delta ^{-6}$
\begin{align}\label{eq:UObs}
 \big||U_t|^2/n-|U_{t+1}|\big|\le \delta |U_{t+1}| / 4.
\end{align}
The base case of \eqref{eq:recU} follows directly from \eqref{eq:baseU}. For the induction step we apply \eqref{eq:auxassumption} together with the induction hypothesis and get 
 \begin{align*}
\Delta_{t_2+s+1}&\le \left(\frac{|U_{t_2+s}|^2}{n}\right)^{3/5}+\left (2\frac{|U_{t_2+s}|}{n}+\Delta_{t_2+s}/n\right )\Delta_{t_2+s}\\
&\le \left(\frac{|U_{t_2+s}|^2}{n}\right)^{3/5}+(2+\delta/2)\frac{|U_{t_2+s}|}{n}\Delta_{t_2+s}\\
&\le \left(\frac{|U_{t_2+s}|^2}{n}\right)^{3/5}+ (2+\delta/2)(2+\delta)^{s}\cdot \frac{|U_{t_2+s}|^2}{n}\cdot n^{-1/11}.
\end{align*}
Using the assumption that $|U_{t_2+s+1}|\ge n^{1/4}$ we get, as $(1/4)\cdot(2/5)- 1/11=1/110$, that\\ $$\big(|U_{t_2+s}|^2/n\big)^{3/5} /\big(|U_{t_2+s}|^2/n\cdot n^{-1/11}\big) \le \delta/4\quad \text{for all}\quad s\in \mathbb N_0\quad \text{and}\quad  n>\delta ^{-111}.$$ This and \eqref{eq:UObs} implies \eqref{eq:recU}.
\end{proof}

\subsection{Proof of Lemma \ref{Lemma3}}\label{SubSLemma3}
We will prove this lemma in two steps, first we start with a simple lemma showing that, if there are much more than $\sqrt n$ uninformed vertices remaining, \emph{pull} will not end in the next round. Moreover if there are substantially less than $\sqrt{n}$ uninformed the protocol will end in the next round.
\begin{lemma}\label{Lemma3.1}
Let $t,t'\in \mathbb{N}$ such that  $|U_{t}|\le \sqrt{n}/\ln n$ and $|U_{t'}|\ge \sqrt{n}\ln n$. Then 
$$P(|U_{t+1}|=0\mid I_t)=o(1) \quad \text{and}\quad P(|U_{t'+1}|>0\mid I_{t'})=o(1).$$
\end{lemma}
\begin{proof}
Note that
\begin{align*}
\mathbb{E}\big[|U_{t+1}|\bigm\vert  I_t\big]=\frac{|U_{t}|^2}{n}\le \ln ^{-2} n.
\end{align*}
This yields with  Markov's inequality the claim for $t$.
To see the claim for $t'$ we observe first that the probability of one uninformed vertex $v\in U_{t'}$ being informed in the next round is
\begin{align*}
P\big(u\in I_{t'+1}\bigm\vert  u\in U_{t'}\big)= \frac{|I_{t'}|}{n}=\left (1-\frac{|U_{t'}|}{n}\right ).
\end{align*}
If $|U_{t'+1}|=0$, then all $|U_{t'}|$ uninformed vertices need to be informed in the next round, and as that happens independently, we get
\begin{align*}
P\big(|U_{t'+1}|=0\bigm\vert  I_{t'}\big)=\left (1-\frac{|U_{t'}|}{n}\right )^{|U_{t'}|}\le e^{-|U_{t'}|^2/n}\le e^{-\ln ^{2} n}. 
\end{align*} 
\end{proof}
Besides the two cases that we considered in the previous lemma a third case is also possible. Indeed, if there are about $\sqrt{n}$ uninformed vertices, the process ending in the next round may happen with some non-trivial probability. The next lemma shows that, however, that this is very unlikely to happen. This means that once the process crosses the threshold of $\sqrt{n}$ it will terminate in the next round with high probability.

\begin{lemma}\label{Lemma3.2}  With high probability for all $t\in \mathbb N$, $$|U_{t}|\notin\Big[\sqrt{n}/ \ln n,\sqrt{n}\ln n\Big].$$
\end{lemma}
\begin{proof}
Let $t_1=\lfloor \log_2(n^{1/3})\rfloor$  and consider the events 
\begin{align}\tag{Event 1}\label{eq:ProbIt1}
\left\{|I_{t_1}|= \Theta\big(n^{1/3}\big)\right\},
\end{align}
and with $\eta = \eta(n) =\log_2 n + \log_2\ln n-\lfloor\log_2 n+ \log_2\ln n\rfloor$,
\begin{align}\tag{Event 2}\label{eq:ProblogIt1}
\left\{\log_2\big(2^{-t_1}|I_{t_1}|\big)\notin \bigcup_{k\in \mathbb{Z}}\left [k+\eta-\frac{3\log_{2}\ln (n^2)}{\ln n},k+\eta+\frac{3 \log_{2}\ln (n^2)}{\ln n}\right ]\right\},
\end{align}
as well as
\begin{align}\tag{Event 3}\label{eq:Lemma2}
\bigcap_{~t\ge 0} \left \{\Big ||U_{t_1+t}|-\big (|U_{t_1}|/n\big )^{2^t}n\Big |\le |U_{t_1+t}| \cdot n^{-1/50}+n^{1/4}\right \}.
\end{align}
All these events occur with high probability. 
For \eqref{eq:ProbIt1} this was already established in Corollary~\ref{ProbCond}.
To see the claim for \eqref{eq:ProblogIt1} we observe that $\log_2\big(2^{-t_1}|I_{t_1}|\big)$ converges to the  continuous random variable $\log_2H$, see Lemma \ref{Lemma1}, and that the right side in \eqref{eq:ProblogIt1} converges to a $\log_2H$ null-set. \eqref{eq:Lemma2} was handled in  Lemma \ref{Lemma2}.

In the remainder of this proof we condition on these events, that is, we assume that $(|I_t|)_{t \ge t_1}$ (and thus also $(|U_t|)_{t \ge t_1}$) are sequences of numbers with the aforementioned properties.

Observe that for $n^{1/3}<a<b<n-n^{1/3}$, \eqref{eq:ProbIt1} and \eqref{eq:Lemma2} imply for $t\ge t_1$ and $n$ large enough
$$
(|U_{t_1}|/n)^{2^{t-t_1}}n\notin [a/2,2b] 
\quad \Longrightarrow\quad 
|U_t|\notin [a,b];
$$
to see this, note that by assumption $|U_t| = (1+o(1))(|U_{t_1}|/n)^{2^{t-t_1}}n$ and so for large enough $n$ if $|U_t|\in [a,b]$, then with room to spare $(|U_{t_1}|/n)^{2^{t-t_1}}n\in [a/2,2b]$.

Set $u_t=(1-|I_{t_1}|/n)^{2^{t-t_1}}n$ and define $T_u=\min\{t\in \mathbb{N}: u_t< 2\sqrt{n} \ln n\}$ as well as $T_\ell=\min\{t\in \mathbb{N}: u_t< \sqrt{n} /(2\ln n)\}$. 
With these definitions we can apply the implication we just derived to the event in the statement of the lemma and obtain that
\begin{align}\label{eq:implication}
T_u = T_\ell
\quad \Longrightarrow \quad 
|U_t|\notin\big[\sqrt{n}/\ln n,\sqrt{n}\ln n\big].
\end{align}
Observe that \eqref{eq:ProbIt1} implies that
$$
    u_{\lfloor(4/3)\log_2n \rfloor}n
    = \big (|U_{t_1}|/n\big )^{2^{\lfloor(4/3)\log_2n \rfloor - t_1}} n 
    = o(1),
$$
thus  $T_u,T_\ell< (4/3)\log_2 n$ and consequently we need to study $u_t$ for that range of $t$ only. Therefore, using $1-x=e^{-x+O(x^2)}$ for small $x$ and $0\le t <(4/3)\log_2 n$
\begin{align*}
u_t&= n\big(1-|I_{t_1}|/n\big)^{2^{t-t_1}}=n\cdot \exp\big (-2^{t-t_1}|I_{t_1}|/n+O(2^{t-t_1}|I_{t_1}|^2/n^2)\big )\\
&=\big(1+O(n^{-2/3})\big)\cdot n\cdot \exp \left (-2^{t-t_1}|I_{t_1}|/n\right ).
\end{align*}
To determine $T_u$ and $T_\ell$ it suffices to solve the equations
\begin{align*}
\exp \left (-2^{t-t_1}|I_{t_1}|/n\right )= \big(1+O(n^{-2/3})\big)c_n n^{-1/2},\quad \text{for } c_n\in \big \{2\ln n,\ 1/2\ln n\big \}.
\end{align*}
Applying logarithms twice and using the first order expansion $\ln (1+x)=x+O(x^2),\ |x|< 1/2$, we readily obtain that  
\begin{align*}
 T_u = \Big\lfloor\log_2n-\log_2(2^{-t_1}|I_{t_1}|)+\log_2\ln n-1-\frac{2 \log_{2}\ln (n^2)}{\ln n}+O(1/\ln n)\Big\rfloor
\end{align*}
and
\begin{align*}
T_\ell = \Big\lfloor\log_2n-\log_2(2^{-t_1}|I_{t_1}|)+\log_2\ln n-1+\frac{2\log_{2}\ln (n^2)}{\ln n}+O(1/\ln n)\Big\rfloor .
\end{align*}
Observe that for any $x,y,z\in [0,1]$ with $x\le y$ it holds that $\lfloor x+z\rfloor\neq \lfloor y+z\rfloor$ if and only if $z\in [1-y,1-x)$; to see this just note that if $z < 1-y$, then both terms are equal to 0, if $z \ge 1-x$ then both terms are equal to 1 and otherwise just one of them is 0.
Thus for $\eta=\log_2 n+ \log_2\ln n-\lfloor\log_2 n+\log_2\ln n\rfloor  $  and $n$ large enough
\begin{align*}
\log_2\big(2^{-t_1}|I_{t_1}|\big)\notin \bigcup_{k\in \mathbb{Z}}\left [k+\eta-\frac{3\log_{2}\ln (n^2)}{\ln n},k+\eta+\frac{3 \log_{2}\ln (n^{2}) }{\ln n}\right ]
\quad \Longrightarrow\quad 
T_u = T_\ell. 
\end{align*}
Since we have assumed~\eqref{eq:ProblogIt1}
we have just established that $T_u = T_\ell$, and together with~\eqref{eq:implication} the proof is completed.
\end{proof}
\phantomsection

\bibliographystyle{abbrv}
\bibliography{literature}
\end{document}